\documentclass[11pt]{amsart}
\usepackage{amsmath,amsthm,amsfonts,latexsym,amssymb}
\usepackage{graphicx}
\usepackage[arrow, matrix, curve]{xy}

\calclayout\evensidemargin .1in

\newcommand{\Z}{\mathbb{Z}}

\newcommand{\Q}{\mathbb{Q}}

\newtheorem{theorem}{Theorem}[section]
\newtheorem{lemma}[theorem]{Lemma}

\newtheorem{corollary}[theorem]{Corollary}

\theoremstyle{definition}

\title[Tight contact structures on hyperbolic Three-manifolds]{Tight contact structures on hyperbolic Three-manifolds}

\author{M. Firat Arikan}
\address{Dept. of Mathematics, Middle East Technical University, Ankara, TURKEY}
\email{farikan@metu.edu.tr}

\author{Merve Se\c{c}g\.in}
\address{Dept. of Mathematics, Uluda\u{g} University, Bursa, TURKEY}
\email{msecgin@uludag.edu.tr}

\date{\today}

\begin{document}

\begin{abstract}
We show the existence of tight contact structures on infinitely many hyperbolic three-manifolds obtained via  Dehn surgeries along sections of hyperbolic surface bundles over circle.
\end{abstract}

\maketitle

\section{Introduction} \label{sec:Introduction}
A \emph{contact three-manifold} is a pair $(M,\xi)$ where $M$ is a smooth $3-$manifold and $\xi\subset TM$ is a totally non-integrable $2$-plane field distribution on $M$. Here we always assume that $\xi$ is a \emph{co-oriented} \emph{positive} contact structure, that is, $\xi=\textrm{Ker}(\alpha)$ for a \emph{contact} $1-$form $\alpha$ satisfying $\alpha \wedge d\alpha> 0$ with respect to a pre-given orientation on $M$. A disk $D$ in a contact $3-$manifold $(M,\xi)$ is called \emph{overtwisted} if the boundary circle $\partial D$ is tangent to $\xi$ everywhere. A contact structure $\xi$ is called overtwisted if there is an \emph{overtwisted} disk in $(M,\xi)$, otherwise it is called \emph{tight}.  It is known that every closed oriented $3-$manifold admits an overtwisted contact structure (\cite{El}, \cite{Ma}). On the other hand, there are $3-$manifolds which do not admit a tight contact structure \cite{EH}.\\

There are some classification results on tight contact structures with respect to the geometric speciality of $3-$manifolds. Lisca and Stipsicz in \cite{LS} proved that a closed oriented Seifert fibered $3-$manifold admits a tight contact structure if and only if it is not gotten $(2q-1)-$surgery along the $(2,2q+1)$ torus knot in $S^3$ for $q\geq 1$. In two independent work (\cite{Col}, \cite{HKM}), they showed the existence of tight contact structures on toroidal $3-$manifolds. It is known that every irreducible $3-$manifold that is neither toroidal nor Seifert fibered is hyperbolic. Kaloti and Tosun in \cite{KT} find infinitely many hyperbolic rational homology spheres admitting tight contact structures. Etg\"u in \cite{Et} also explored that infinitely many hyperbolic $3-$manifolds that carry tight contact structures. His construction uses Dehn surgeries along sections of hyperbolic torus bundles over $S^1$. Here we'll follow similar ideas for surface bundles over $S^1$ with fiber genus at least two. \\

Let $\Sigma_g$ be a closed connected orientable surface with genus $g$. In this paper assume that $g$ is always greater than 1. We will denote $MCG(\Sigma_g)$ by the \emph{mapping class group} of $\Sigma_g$, i.e, the group of isotopy classes of orientation preserving homeomorphisms of $\Sigma_g$. Let $t_a$ be the positive Dehn twist along a simple closed curve $a$.\\

Let $\phi \in MCG(\Sigma_g)$ be the mapping class representing the homeomorphism

\begin{equation} \label{eqn:homeomorphism}
t_{a_1}^mt_{a_2}\cdots t_{a_{2g}}t_{a_{2g+1}}^n
\end{equation}
where $a_i$'s are simple closed curves on $\Sigma_g$ as indicated in Figure \ref{fig:Fig_1}.

\begin{figure}[ht]
	\begin{center}
		\includegraphics [scale=.61]{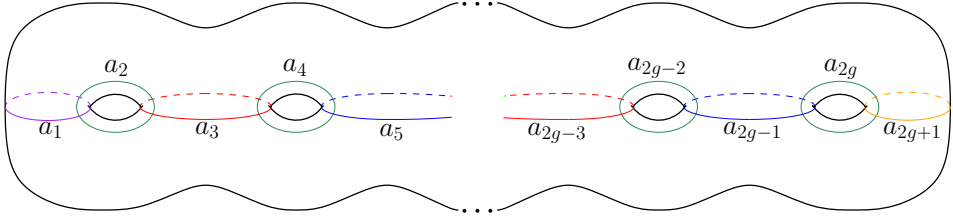}
		\caption{Simple closed curves on the surface $\Sigma_g$.}
		\label{fig:Fig_1}
	\end{center}
\end{figure}

Denote by $M_\phi$ the \emph{mapping torus} with fibers $\Sigma_g$ and monodromy $\phi$.  Let $M_\phi(r)$ be the surgered manifold obtained by performing rational $r$-surgery along a section of $M_\phi$. Clearly, $\phi$ has a fixed point, so such a section exists. The following theorems give examples required:

\begin{theorem} \label{thm:Main_Theorem1}
Suppose $g\geq2$, $m,n \in \Z$, $r \in \Q$ and $\phi$ as indicated in $(\ref{eqn:homeomorphism})$. Then
 $M_\phi(r)$ is hyperbolic for all but finitely many $m$ and $r$.

\end{theorem}

\begin{theorem} \label{thm:Main_Theorem2}
Suppose $g\geq 1$, $r \in \Q$ and $\phi$ as indicated in $(\ref{eqn:homeomorphism})$. Then
$M_\phi(r)$ admits a tight contact structure $\xi$ for any  $m,n \in \Z^+$ and for all\, $r\neq 2g-1$.
\end{theorem}

As a consequence of the theorems we have:
\begin{corollary} \label{cor:Main_Theorem}
Suppose $g\geq 2$, $m,n \in \Z^+$, $r \in \Q$ and $\phi$ as indicated in $(\ref{eqn:homeomorphism})$. Then
$M_\phi(r)$ is a hyperbolic manifold admitting a tight contact structure for all\, $r\neq 2g-1$ and all but finitely many $m \in \Z^+$. \qed
\end{corollary}

The proof of Theorem \ref{thm:Main_Theorem1} and Theorem \ref{thm:Main_Theorem2} will be given in Section \ref{sec:Proof of Theorem1} and Section \ref{sec:Proof of Theorem2}.



\section{Proof of Theorem 1.1} \label{sec:Proof of Theorem1}

In order to prove the theorem, we'll focus on pseudo-Anosov homeomorphisms and construct infinitely many hyperbolic $3-$manifolds via pseudo-Anosov monodromies. A hyperbolic $3-$manifold is a $3-$manifold which admits a complete finite-volume hyperbolic structure. Thurston \cite{T} demonstrated that an orientable surface bundle over circle whose fiber is a compact surface of negative Euler characteristic is hyperbolic if and only if the monodromy of the surface bundle is a pseudo-Anosov homeomorphism. Another deep result of Thurston is hyperbolic Dehn surgery theorem which states that a hyperbolic $3-$manifold remains hyperbolic after Dehn filling along a link for all slopes except finitely many of them (For details see \cite{T2}). In order to apply these results, we need a lemma where we construct infinitely many pseudo-Anosov diffeomorphisms as products of certain Dehn twists:

\begin{lemma} \label{lem:Pseudo-Anosov}
Let $\phi$ be the class in $MCG(\Sigma_g)$ as described in (\ref{eqn:homeomorphism}) above. Then $\phi$ is pseudo-Anosov for any integer $n$ and for all but at most $7$ consecutive values of $m$.
\end{lemma}

Denote by $\iota(\alpha,\beta)$ \emph{geometric intersection number} of the curves $\alpha$ and $\beta$. We say a set of simple closed curves $\{\gamma_1,\gamma_2,\ldots,\gamma_k\}$ \emph{fills} $\Sigma_g$ if $\Sigma_{g}\setminus\{\gamma_1,\gamma_2,\ldots,\gamma_k\}$ is a disjoint union of topological disks. In order to prove Lemma \ref{lem:Pseudo-Anosov}, we use the following theorem of Fathi:

\begin{theorem} $($\cite{Fa}$)$ \label{thm:Pseudo-Anosov}
Let $f$ be the class in $MCG(\Sigma_g)$ and let $\gamma$ be a simple closed curve in $\Sigma_g$. If the orbit of $\gamma$ under $f$ fills $\Sigma_g$, then $t_\gamma^mf$ is a pseudo-Anosov class except for at most 7 consecutive values of $m$.
\end{theorem}
\begin{proof}[Proof of Lemma \ref{lem:Pseudo-Anosov}] Let $\gamma$ represents the curve $a_1$ and let $f$ be the product of Dehn twists $t_{a_1}t_{a_2}\cdots t_{a_{2g}}t_{a_{2g+1}}^n$. Then conclude that
$$f(\gamma)=t_{a_1}t_{a_2}(a_1)=a_2, \quad f^2(\gamma)=t_{a_1}t_{a_2}t_{a_3}(a_2)=a_3,$$
\noindent and inductively,
\begin{center}
$f^i(\gamma)=a_{i+1}$ for all $i \in {1,2,\ldots,2g-1}$.
\end{center}

\medskip
\noindent Since the complement $\Sigma_g \setminus \{a_1,\ldots,a_{2g}\}$ is a topological disk, we can say the orbit of $\gamma$ under $f$ fills $\Sigma_g$.  As a result of Theorem \ref{thm:Pseudo-Anosov}, $\phi$ is pseudo-Anosov except for at most $7$ consecutive $m$ values.
\end{proof}

Now we have a family of pseudo-Anosov monodromies. Using \cite{T} we can say that the surface bundles $M_\phi$ are all hyperbolic. By hyperbolic Dehn surgery theorem the surgered manifolds $M_\phi(r)$ are hyperbolic for all $m, n \in \Z$ and $r\in \Q$ except 7 values of $m$ and finitely many ``bad'' slopes $r$. This finishes the proof of Theorem \ref{thm:Main_Theorem1}.
\qed


\begin{figure}[ht]
		\includegraphics[scale=.6]{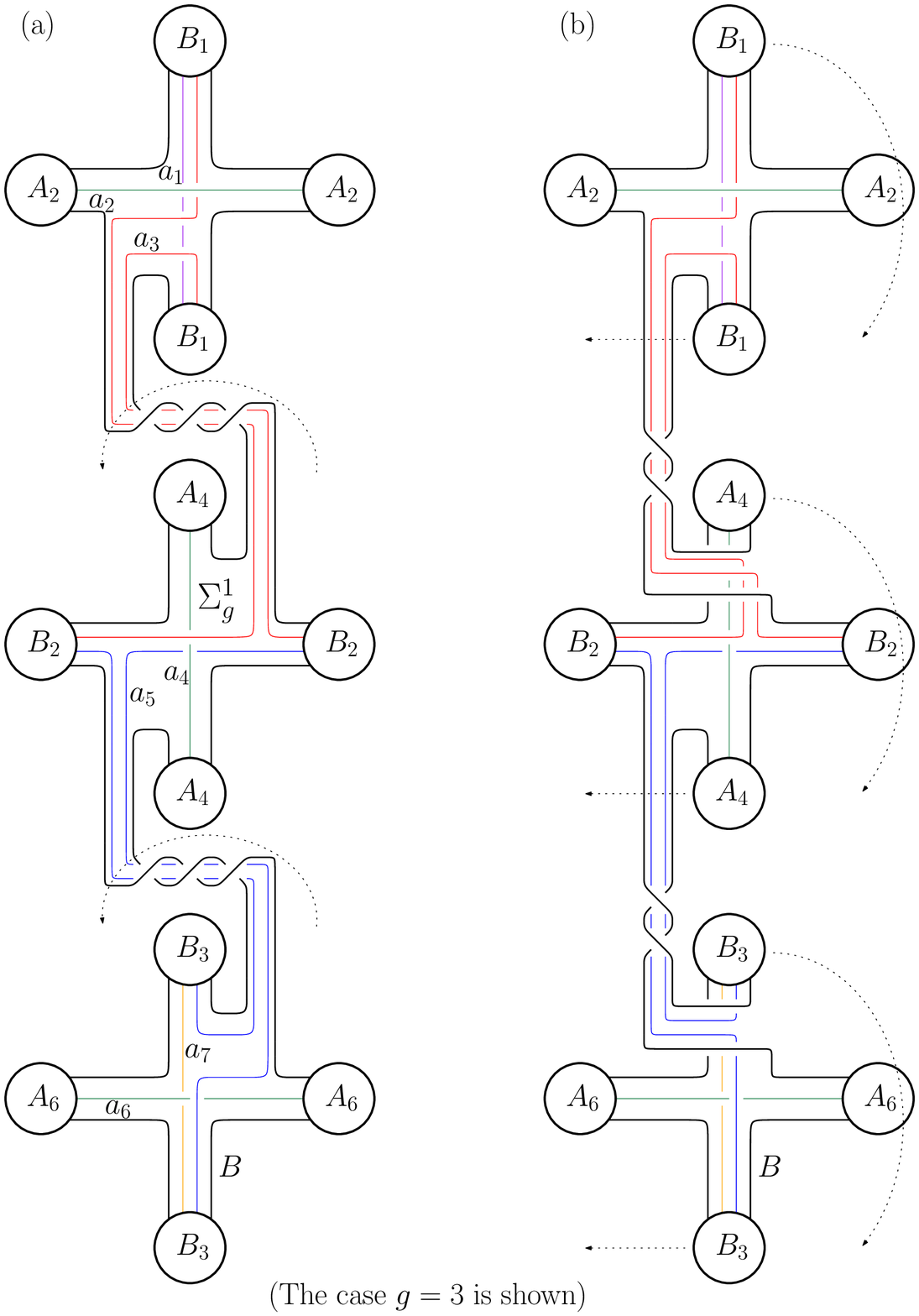}
		\caption{
}
		\label{fig:Four_Manifold_Yeni}
\end{figure}

\section{Proof of Theorem 1.2} \label{sec:Proof of Theorem2}

We will analyze the proof with respect to the parity of the genus $g$ of the fiber $\Sigma_g$. First assume $g\geq 3$ odd. Note that conjugation of the monodromy by any class of $MCG(\Sigma_g)$ does not change the mapping torus up to diffeomorphism. Since
$$t_{a_2}\cdots t_{a_{2g}}t_{a_{2g+1}}^nt_{a_1}^m=t_{a_1}^{-m} \phi \,t_{a_1}^m$$

\noindent we may replace $\phi$ in (\ref{eqn:homeomorphism}) with the mapping class $t_{a_2}\cdots t_{a_{2g}}t_{a_{2g+1}}^nt_{a_1}^m$. Also observe that $M_\phi(r)$ can be also obtained from a Dehn surgery on the binding of an open book decomposition whose page is $\Sigma_g^1$ (punctured $\Sigma_g$) and monodromy can be still assumed to be $\phi \in MCG(\Sigma_g^1)$. We will construct the required contact structure $\xi$ on $M_\phi(r)$ via Dehn surgery on the open book decomposition $(\Sigma_g^1,\phi)$ along its binding.

It is known (see \cite{AO}, \cite{Gi}) that the contact structure, say $\xi_0$, (before the surgery along binding) supported by $(\Sigma_g^1,\phi)$ is Stein fillable. More precisely, consider the handlebody diagram of the smooth $4-$manifold $X_\phi$ given in Figure \ref{fig:Four_Manifold_Yeni}-(a) (in the case of genus 3) with ``$2g$'' $1-$handles and ``$m+n+2g-1$'' $2-$handles. Note that Figure \ref{fig:Four_Manifold_Yeni}-(a) describes a Lefschetz fibration structure on $X_\phi$ with a regular fiber $\Sigma_g^1$ and the vanishing cycles $a_1, a_2, ..., a_{2g+1}$. There are $n$ copies for $a_{2g+1}$ and $m$ copies for $a_1$ (not drawn for simplicity). All coefficients (except on $B$) are $-1$ with respect to the framing given by the page $\Sigma_g^1$. We remark that no handle is attached along the binding of the induced open book $(\Sigma_g^1,\phi)$ on the boundary $\partial X_\phi$ which is realized as $B$ in the figure.

\begin{figure}[ht]
		\includegraphics[scale=.6]{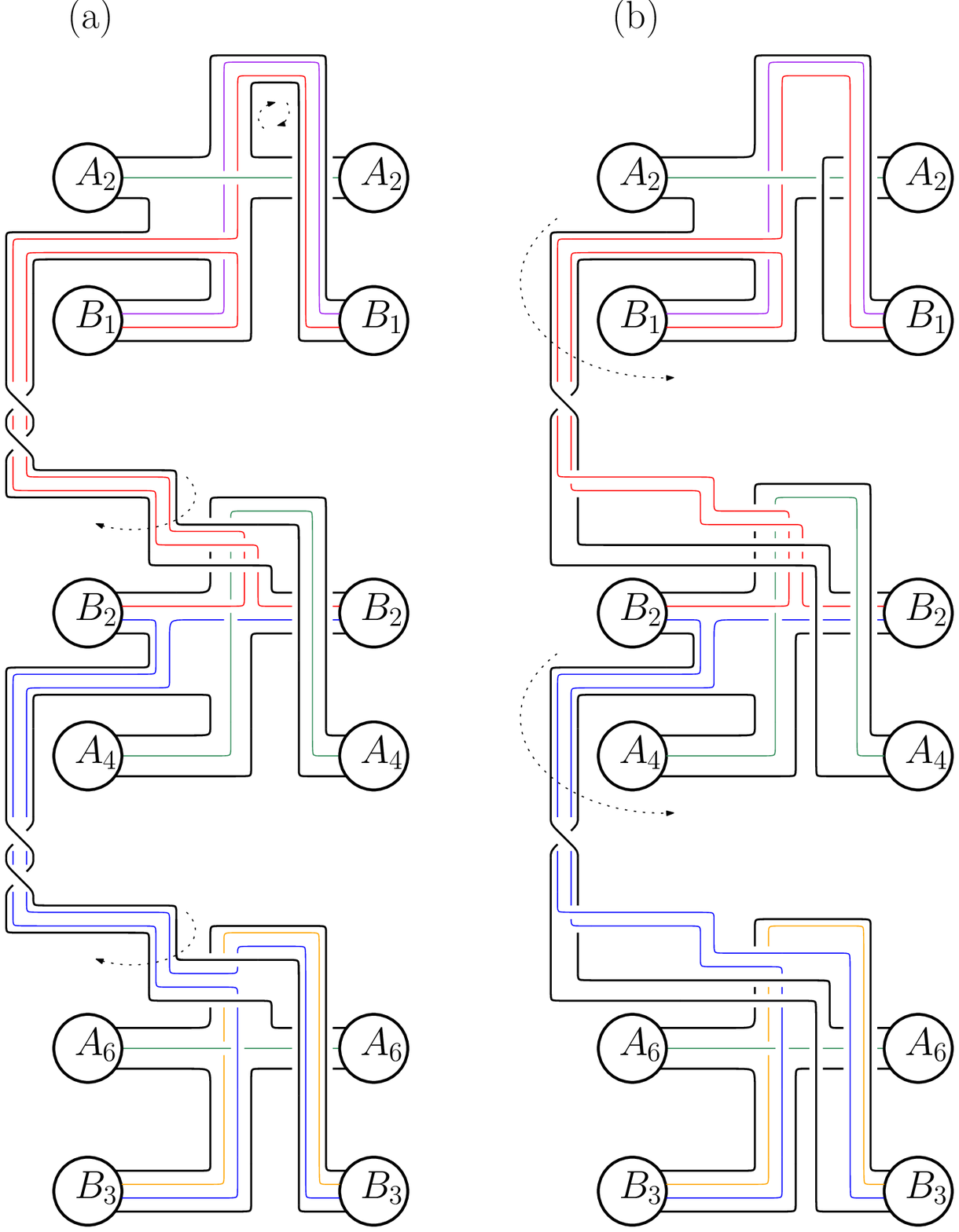}
		\caption{
}
		\label{fig:Four_Manifold_Yeni_2}
\end{figure}

Next starting from the topological description in Figure \ref{fig:Four_Manifold_Yeni}-(a) of $X_\phi$, we'll get a diagram describing a Stein structure on $X_\phi$ inducing $\xi_0$ as follows: First we flip the twisted bands over the $1-$handles as pointed out in Figure \ref{fig:Four_Manifold_Yeni}-(a) and get Figure \ref{fig:Four_Manifold_Yeni}-(b). Figure \ref{fig:Four_Manifold_Yeni_2}-(a) gives another handle description of $X_\phi$ obtained by moving the feet of $1$-handles as indicated by the dotted arrows in Figure \ref{fig:Four_Manifold_Yeni}-(b). Then flip the bands as shown in Figure \ref{fig:Four_Manifold_Yeni_2}-(a) to get rid of one more left half twist for each band (see Figure \ref{fig:Four_Manifold_Yeni_2}-(b)), and obtain Figure \ref{fig:Four_Manifold_Yeni_3}-(a) by flipping the connecting bands over the feet of $1$-handles suggested by the dotted arrows in Figure \ref{fig:Four_Manifold_Yeni_2}-(b).  Figure \ref{fig:Four_Manifold_Yeni_3}-(b) defines a Stein structure on $X_\phi$ obtained by putting the attaching circles in part (a) into Legendrian positions, where a Legendrian realization $L_0$ of $B$ in the tight contact boundary $\partial X_\phi$ is also provided. All coefficients (except on $L_0$) are $-1$ with respect to Thurston-Bennequin (contact) framing in $\partial X_\phi$ and no handle is attached along $L_0$. Note that $tb(L_0)=2$ (the case $g=3$ is shown). In the general case, $tb(L_0)=g-1$. Finally, we use the trick (``Move 6'') in Figure 20 of \cite{Go} to obtain a Legendrian representation $L$ of $B$ with $tb(L)=2g-1$ (see Figure \ref{fig:Stein_Manifold_g_3}). Note that   Figure \ref{fig:Stein_Manifold_g_3} describes the same Stein structure on $X_\phi$ as in Figure \ref{fig:Four_Manifold_Yeni_3}-(b).

\begin{figure}[ht]
	\begin{center}
		\includegraphics[scale=.6]{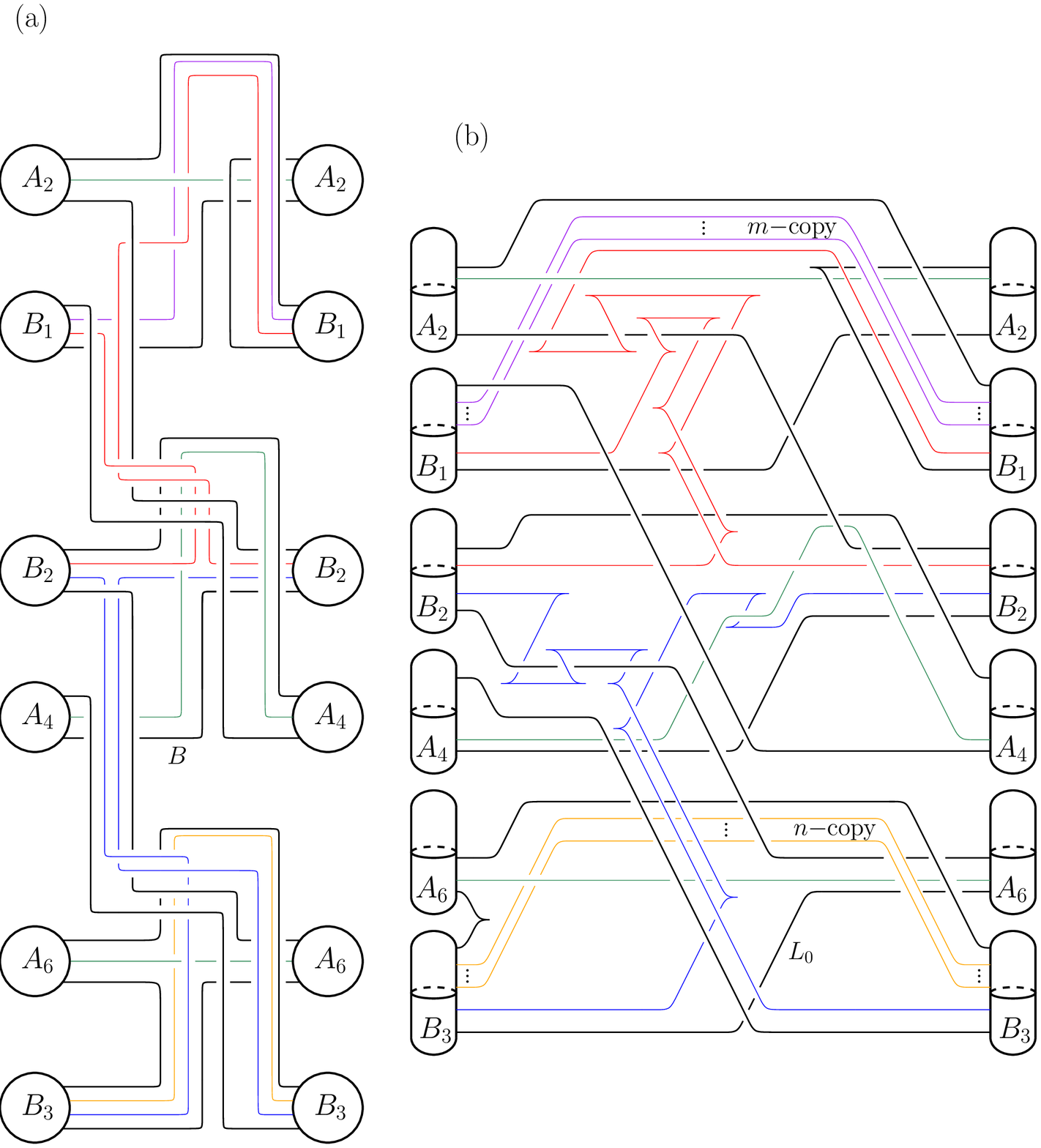}
		\caption{
}
		\label{fig:Four_Manifold_Yeni_3}
	\end{center}
\end{figure}

\newpage

\begin{figure}[ht]
	\begin{center}
		\includegraphics[scale=.6]{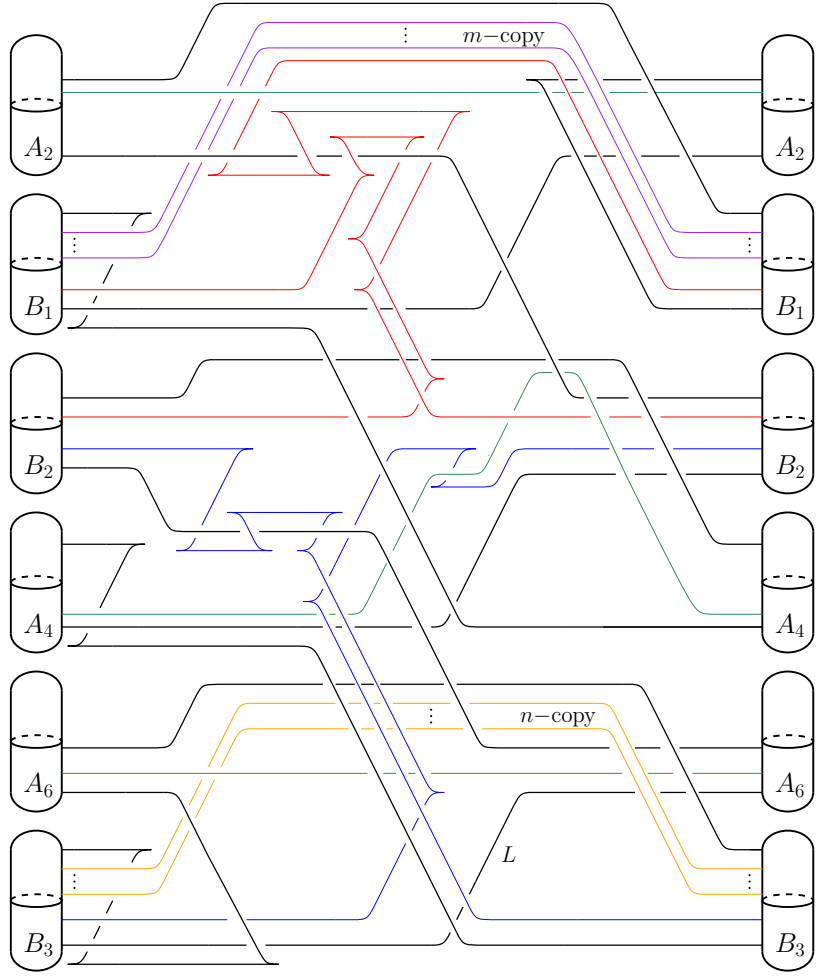}
		\caption{The same Stein structure on $X_\phi$ as in Figure \ref{fig:Four_Manifold_Yeni_3}-(b), and another Legendrian realization $L$ of the binding $B$ in the tight contact boundary $\partial X_\phi$. $L$ is obtained from $L_0$ by applying ``Move 6'' (smooth but non-Legendrian isotopy of $L_0$) $g$ times using the left foot of the corresponding 1-handles (when $g=3$, handles are $B_1, A_4, B_3$). All coefficients (except on $L$) are $-1$ with respect to Thurston-Bennequin (contact) framing in $\partial X_\phi$. No handle attached along $L$. Note that $tb(L)=5$ (the case $g=3$ is shown). In the general case, $tb(L)=2g-1$.}
		\label{fig:Stein_Manifold_g_3}
	\end{center}
\end{figure}

Now if $g\geq 2$ is even, we replace the monodromy $\phi$ with $t_{a_{2g+1}}^n t_{a_2}\cdots t_{a_{2g}}t_{a_1}^m$ since
$$t_{a_{2g+1}}^n t_{a_1}^{-m}\phi\,t_{a_{2g+1}}^{-n} t_{a_1}^m=t_{a_{2g+1}}^n t_{a_2} \cdots t_{a_{2g}} t_{a_1}^m.$$
Then starting from the handlebody diagram given in Figure \ref{fig:Stein_Manifold_g_4}-(a) (where the case $g=4$ is shown) and following the moves
as in the case of odd genus, one can get Figure \ref{fig:Stein_Manifold_g_4}-(b) describing a Stein structure realizing a Legendrian representation $L$ with $tb(L)=2g-1$ as in Figure \ref{fig:Stein_Manifold_g_3}. One should note that we need to consider different monodromies (but still giving the same mapping torus) depending on the parity of $g$ to make the contact and the page framing on any attaching circle coincide.\\

\begin{figure}[ht]
	\begin{center}
		\includegraphics[scale=.6]{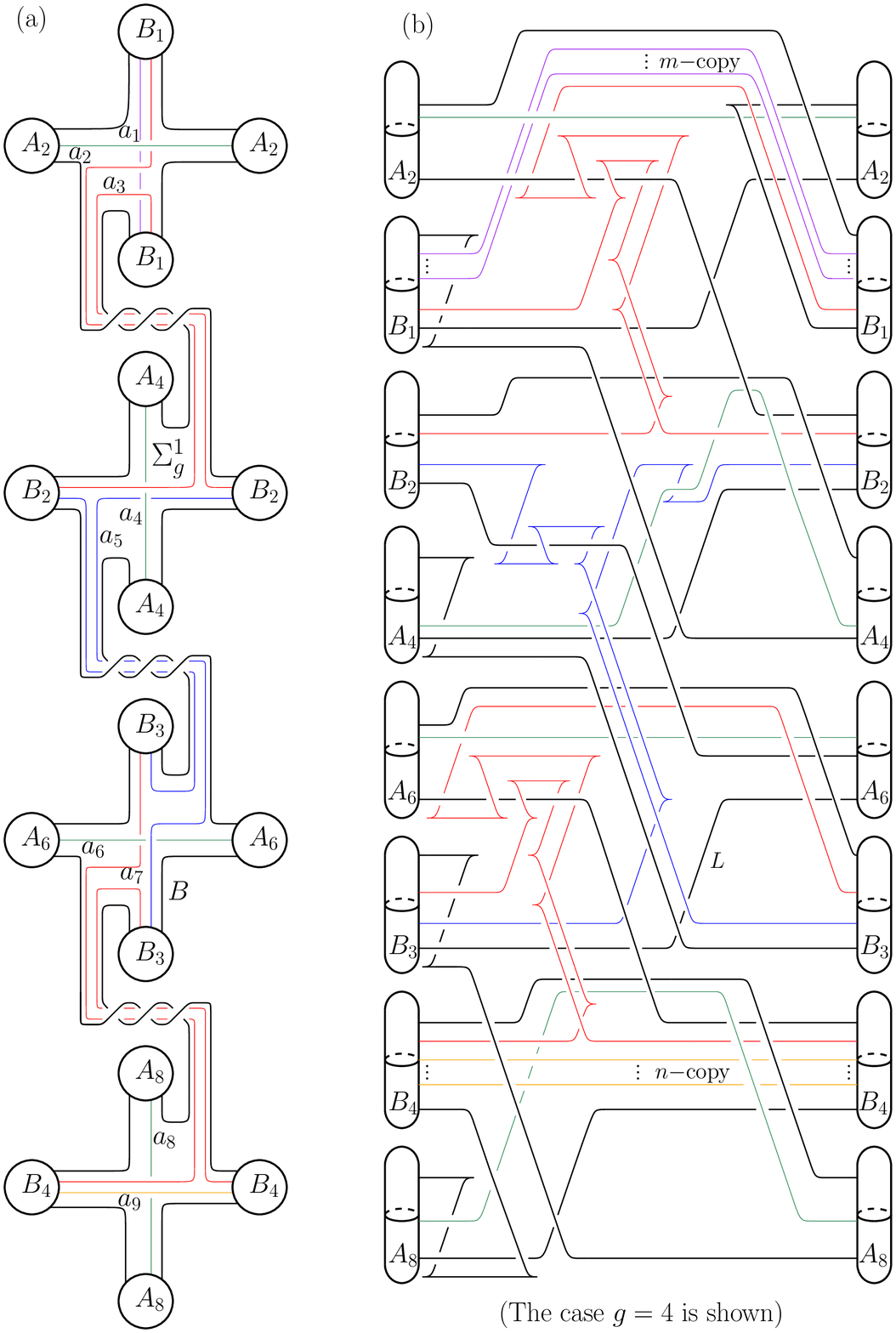}
		\caption{}
        \label{fig:Stein_Manifold_g_4}
	\end{center}
\end{figure}

Now (in any case of $g$) we first (Legendrian) slide (Stein) $2-$handle corresponding $a_3$ over the ones represented by the curves $a_1, a_5, a_7, ..., a_{2g+1}$, and then cancel the $2-$handles represented by $a_5, a_7,...,a_{2g-1}$ with the corresponding $1-$handles. Second, we (Legendrian) slide $2-$handles represented by the curves $a_1$ and $a_{2g+1}$ over a fixed one (chosen from each family in Figure \ref{fig:Stein_Manifold_g_3} / Figure \ref{fig:Stein_Manifold_g_4}-(b)), and then cancel $1-$handles $B_1$ and $B_{g}$ with the chosen $2-$handles corresponding $a_1$ and $a_{2g+1}$ respectively. Also we cancel each $1-$handle $A_i$ with the $2-$handle corresponding the curve $a_i$ for each $i$ even. As a result, we obtain another (but equivalent) Stein structure on $X_\phi$ which can be also considered as the contact surgery diagram for $\xi_0$ on $\partial X_\phi$. Finally, we set $r'=r-2g+1$ and perform $r'-$ contact surgery along $L \subset (\partial X_\phi, \xi_0)$ to get a contact structure $\xi$ on $M_\phi(r)$ whose diagram is given in Figure \ref{fig:Contact_Structure_g_3} (where we use continued fractions).
\begin{figure}[ht]
	\begin{center}
		\includegraphics[scale=.56]{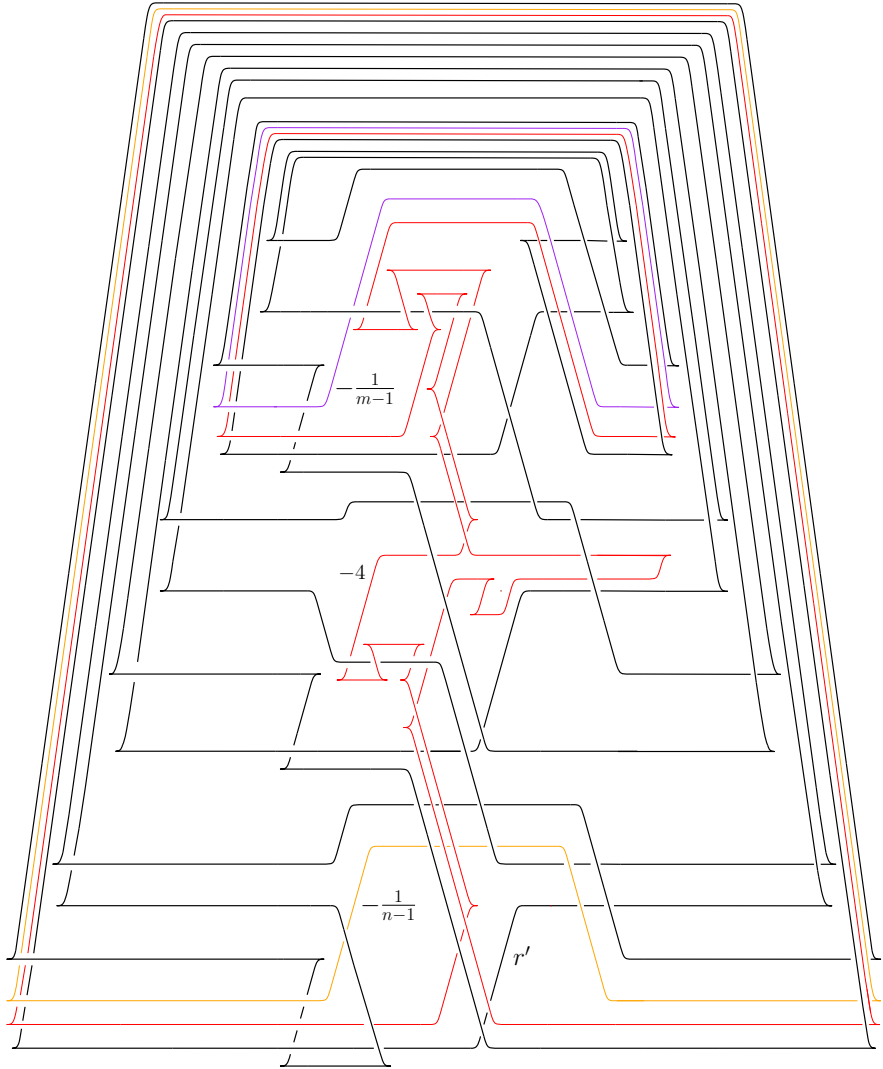}
		\caption{The contact $3-$manifold $(M_\phi(r), \xi)$. (The case $g=3$ is shown.)}
		\label{fig:Contact_Structure_g_3}
	\end{center}
\end{figure}

First suppose $r'=r-2g+1<0$. We know any contact surgery with negative contact framing can be converted to a sequence of contact $(-1)-$surgeries and $(-1)-$surgeries preserve Stein fillability (\cite{DGS}, \cite{Eliashberg}, \cite{Gi}). Thus $(M_\phi(r), \xi)$  is Stein fillable (hence tight).\\

Now let $r'=r-2g+1>0$. By Thurston-Winkelnkemper construction (\cite{TW}), it is known that the binding $B$ is transverse to the contact structure supported by the open book decomposition. Also since $\partial X_\phi$ is Stein fillable, $\xi_0$ has nonzero contact invariant \cite{OSS}. As a result of Conway's work (see \cite{Co}, Theorem 1.6) if $K$ is a fibered transverse knot in a contact $3-$manifold $(M,\eta)$ where $\eta$ has nonvanishing contact class, then r-surgery along $K$ preserves the non-vanishing of the contact class if $r>2g-1$ where $g$ is the genus of $K$. Hence we conclude that $(M_\phi(r), \xi)$ has nonzero contact invariant (hence tight) through Conway's result. This finishes the proof of Theorem \ref{thm:Main_Theorem2}.
\qed

\medskip \noindent {\em Acknowledgments.\/} The authors would like to
thank James Conway and Mustafa Korkmaz for their invaluable comments.



\end{document}